\documentclass[a4paper, 12pt]{amsart}
\usepackage[top=40truemm,bottom=40truemm,left=35truemm,right=35truemm]{geometry}
\usepackage{amssymb,amsmath}
\usepackage{longtable}

\usepackage{color}
\usepackage[normalem]{ulem}

\makeatletter
\@namedef{subjclassname@2010}{
  \textup{2010} Mathematics Subject Classification}
\makeatother

\numberwithin{equation}{section}

\newcommand{\A}{\mathbb{A}}

\newcommand{\G}{\mathbb{G}}

\newcommand{\Z}{\mathbb{Z}}

\newcommand{\Ker}{\operatorname{Ker}}

\newcommand{\Spec}{\operatorname{Spec}}
\newcommand{\EXP}{\operatorname{EXP}}
\newcommand{\ML}{\operatorname{ML}}
\newcommand{\AK}{\operatorname{AK}}
\newcommand{\trdeg}{{\rm tr.deg}\:}
\newcommand{\ch}{{\rm char}\:}
\newcommand{\lc}{{\rm lc}\:}
\newtheorem{thm}{Theorem}[section]
\newtheorem{prop}[thm]{Proposition}
\newtheorem{lem}[thm]{Lemma}
\newtheorem{cor}[thm]{Corollary}
\newtheorem{defn}[thm]{Definition}
\newtheorem{example}[thm]{Example}

\newtheorem*{problem1}{Problem 1}
\newtheorem*{claim1}{Claim 1}
\newtheorem*{claim2}{Claim 2}
\newtheorem*{claim3}{Claim 3}
\begin{document}
\title[Remarks on retracts of polynomial rings in three variables]{Remarks on retracts of polynomial rings in three variables in any characteristic}
\author{Hideo Kojima}
\address[H. Kojima]{Department of Mathematics, Faculty of Science, Niigata University, 8050 Ikarashininocho, Nishi-ku, Niigata 950-2181, Japan}
\email{kojima@math.sc.niigata-u.ac.jp}
\author{Takanori Nagamine}
\address[T. Nagamine]{Department of Mathematics, College of Science and Technology, 
Nihon University, 1-8-14 Kanda-Surugadai, Chiyoda-ku, Tokyo 101-8308, Japan}
\email{nagamine.takanori@nihon-u.ac.jp}
\author{Riko Sasagawa}
\address[R. Sasagawa]{Niigata Prefectural Tokamachi Senior High School, 1-203 Honcho-Nishi, Tokamachi-shi, Niigata, 948-0083, Japan}
\thanks{Research of the first author is supported by JSPS KAKENHI Grant Number JP 21K03200, and the second author by JSPS KAKENHI Grant Number JP 21K13782.}

\begin{abstract}
Let $A$ be a retract of the polynomial ring in three variables over a field $k$. It is known that if $\ch(k) = 0$ or $\trdeg_k A \not= 2$ then $A$ is a polynomial ring. In this paper, we give some sufficient conditions for $A$ to be the polynomial ring in two variables over $k$ when $\ch(k) > 0$ and $\trdeg_k A = 2$. 
\end{abstract}
\maketitle

\setcounter{section}{0}

\section{Introduction}

Throughout the paper, $k$ denotes a field. For an integral domain $R$ and a positive integer $n$, $R^{[n]}$ denotes the polynomial ring in $n$ variables over $R$. $Q(R)$ denotes the quotient field of $R$.

Let $B$ be a $k$-algebra and $A$ its $k$-subalgebra. We say that $A$ is a \textit{retract} of $B$ if there exists an ideal $I$ of $B$ such that $B = A \oplus I$, where the sum is $A$-direct (cf.\ Definition 2.1). 
Several mathematicians have studied retracts of $k$-domains. See, e.g., Costa \cite{C77}, Spilrain--Yu \cite{SY00}, Liu--Sun \cite{LS18}, \cite{LS21}, the second author \cite{N19}, Chakraborty--Dasgupta--Dutta--Gupta \cite{CDDG21},  Gupta-the second author \cite{GN23}. In particular, the study of retracts of polynomial rings has been done because it is closely related to the Zariski cancellation problem.

In \cite{C77}, Costa posed the following problem.

\begin{problem1}
Is every retract of $k^{[n]}$ a polynomial ring?
\end{problem1}

We note that Problem 1 includes the Zariski cancellation problem. Indeed, if a $k$-algebra $A$ satisfies $A^{[1]} \cong k^{[n]}$ as $k$-algebras, then $A$ can be regarded as a retract of $k^{[n]}$. So, if Problem 1 has an affirmative answer for fixed $k$ and $n$, then we conclude that $A \cong k^{[n-1]}$. 

Let $A$ be a retract of $B:= k^{[n]}$. 
We have the following results on Problem 1.
\begin{itemize}
\item[(1)]
If $\trdeg_k A = 0$ (resp.\  $\trdeg_k A = n$), then $A = k$ (resp.\ $A = B$). 
\item[(2)]
(\cite[Theorem 3.5]{C77})\ \  If $\trdeg_k A = 1$, then $A \cong k^{[1]}$. In particular, if $n = 2$, then $A$ is a polynomial ring.
\item[(3)]
(\cite[Theorem 2.5]{N19}, \cite[Theorem 5.8]{CDDG21})\ \ If $\ch(k) = 0$ and $\trdeg_k A = 2$, then $A \cong k^{[2]}$. In particular, if $n = 3$ and $\ch(k) = 0$, then $A$ is a polynomial ring.
\item[(4)]
In the case $n \geq 3$ and $\ch(k) > 0$, Gupta \cite{G14-1}, \cite{G14-2} proved that                                       
the Zariski cancellation problem does not hold for the affine $n$-space $\A^n$. So, Problem 1 has counter examples when $n \geq 4$ and $\ch(k) > 0$. 
\end{itemize}
Therefore, Problem 1 remains open in the following cases:
\begin{itemize}
\item
$n = 3$ and $\ch (k) > 0$.
\item
$n \geq 4$ and $\ch (k) = 0$. 
\end{itemize}

Let $A$ be a retract of $k^{[3]}$ with $\trdeg_k A = 2$. When $\ch (k) = 0$, $A \cong k^{[2]}$ by \cite[Theorem 2.5]{N19} (or \cite[Theorem 5.8]{CDDG21}). Unfortunately, their proofs do not work in the case $\ch (k) > 0$. 
In fact, the second author \cite{N19} used the result \cite[Theorem 1.1 (a)]{K80} in the proof of \cite[Theorem 2.5]{N19}, which does not hold true when $\ch(k) > 0$.
In the proof of \cite[Theorem 5.8]{CDDG21}, the authors used \cite[Theorem 2.8]{CDDG21}, which cannot be used when $\ch (k) > 0$. 
Furthermore, they used the fact that $k^{[2]}$ over a field $k$ of characteristic zero has no non-trivial forms (cf.\ \cite[Theorem 3]{K75}), which does not hold true when $\ch(k) > 0$. 

In this paper, we study the case where $n = 3$ and $\ch (k) > 0$. In Section 2, we recall some results on retracts and exponential maps and their rings of constants on integral domains.
In Section 3, we give some sufficient conditions for a retract $A$ of $k^{[3]}$ to be a polynomial ring in two variables over $k$ when $\ch(k) > 0$ and $\trdeg_k A = 2$.  More precisely, we prove that if one of the following conditions holds, then $A \cong k^{[2]}$: 
\begin{enumerate}
\item $A$ is a ring of constants of some exponential map (cf.\ Corollary 3.2). 
\item $\ML(A) \not= A$, that is, there exists a non-trivial exponential map on $A$ (cf.\ Theorem 3.3). 
\item $k$ is algebraically closed and there exists a $\Z$-grading on $A$ with $A_0\not=A$  (cf.\ Theorem 3.7). 
\end{enumerate}
In Section 4, we give several examples of retracts of low-dimensional polynomial rings over $k$. These examples indicate they are polynomial rings. 

\section{Preliminary results}

Let $R$ be an integral domain. 
In this section, all rings we consider are $R$-domains. 

\begin{defn}
{\rm 
An $R$-subalgebra $A$ of an $R$-algebra $B$ is called a {\em retract} if it satisfies one of the following equivalent conditions:
\begin{itemize}
\item[(1)]
There exists an idempotent $R$-algebra endomorphism $\varphi$ of $B$, which is called a {\em retraction},
such that $\varphi(B) = A$.  
\item[(2)]
There exists an $R$-algebra homomorphism $\varphi: B \to A$ such that $\varphi|_{A} = {\rm id}_A$.
\item[(3)]
There exists an ideal $I$ of $B$ such that $B = A \oplus I$ as $A$-modules. 
\end{itemize}
}
\end{defn}

We recall some properties of retracts.

\begin{lem} \label{properties of retracts}
Let $B$ be an $R$-domain and $A$ a retract of $B$. Then the following assertions hold true.
\begin{itemize}
\item[(1)]
$A$ is algebraically closed in $B$, namely, every algebraic element of $B$ over $A$ belongs to $A$. 
So $Q(A) \cap B = A$. 
\item[(2)]
If $B$ is finitely generated over $R$, then so is $A$. 
\item[(3)]
If $B$ is a UFD, then so is $A$. 
\item[(4)]
If $B$ is regular, then so is $A$. 
\end{itemize}
\end{lem}

\begin{proof}
(1)\ \ See \cite[Lemma 1.3]{C77}.

(2) \ \ Obvious. 

(3)\ \ See \cite[Proposition 1.8]{C77}. 

(4)\ \ See \cite[Corollary 1.11]{C77}. 
\end{proof}

Let $B$ be an $R$-domain and let $\sigma: B \to B^{[1]}$ be an $R$-algebra homomorphism from $B$ to $B^{[1]}$. For an indeterminate $U$ over $B$, let $\sigma_U$ denote the map $\sigma: B \to B[U]$. Then $\sigma$ is called an \textit{exponential map} on $B$ if the following conditions hold.
\begin{itemize}
\item[(i)]
$\epsilon_0 \circ \sigma_U  = {\rm id}_B$, where $\epsilon_0 : B[U] \to B$ is the evaluation at $U = 0$. 
\item[(ii)]
For indeterminates $U, \, V$, we have $\sigma_V \circ \sigma_U = \sigma_{U+V}$, where $\sigma_V: B \to B[V]$ is extended to an $R$-algebra homomorphism $\sigma_V : B[U] \to B[U,\, V] (= B^{[2]})$ by setting $\sigma_V(U) = U$. 
\end{itemize}
For an exponential map $\sigma : B \to B^{[1]}$, we denote the ring of constants of $\sigma$ by $B^{\sigma}$. Namely, $B^{\sigma} = \{ x \in B \ | \ \sigma(x) = x \} = \sigma^{-1}(B)$. 
The exponential map $\sigma$ is said to be \textit{trivial} if $B^{\sigma} = B$. 
We note that the conditions (i) and (ii) as above is equivalent to the condition that $\sigma$ is a co-action of the additive group scheme $\G_a = \Spec R[U]$ on $\Spec B$. If $R$ is a field of $\ch(R) = 0$, then an exponential map $\sigma: B \to B^{[1]}$ is corresponding to a locally nilpotent $R$-derivation $D: B \to B$. In particular, $B^{\sigma} = \Ker D$. 

For each $b \in B$, we define $\deg_{\sigma} (b)$ (resp.\ $\lc_{\sigma} (b)$) to be the degree of $\sigma(b)$ as a polynomial in the indeterminate over $B$ (resp.\ the leading coefficient of $\sigma (b)$). 

Assume that $\sigma$ is non-trivial, i.e., $B^{\sigma} \not= B$. We call $s \in B$ a {\em local slice} of $\sigma$ if $\deg_{\sigma} (s) = \min \{ \deg_{\sigma} (b) \ | \ b \in B \setminus B^{\sigma} \}$. For a local slice $s \in B$ of $\sigma$, it is clear that $\lc_{\sigma} (s) \in B^{\sigma}$. 
We recall the following well-known result.

\begin{lem}
Let $\sigma$ be a non-trivial exponential  map on an $R$-domain $B$. 
For a local slice $s \in B$ of $\sigma$, set $a = \lc_{\sigma} (s)$. 
Then $B_a (:= B[\frac{1}{a}]) = (B^{\sigma})_a [s]$ and $s$ is indeterminate over $B^{\sigma}$. 
\end{lem}

\begin{proof}
See, e.g., \cite[Corollary 2.4]{Kuroda17}. 
\end{proof}

Let $\EXP(B)$ denote the set of all exponential maps on $B$. Then the Makar-Limanov invariant of $B$, which  is denoted by $\ML (B)$, is defined as
$$
\ML (B) = \bigcap_{\sigma \in \EXP (B) } B^{\sigma}.
$$
The Makar-Limanov invariant $\ML(B)$ of $B$ is also called the AK invariant of $B$ (the ring of absolute constants of $B$) and denoted by $\AK(B)$. 

Let $A$ be a subring of $B$. Then $A$ is said to be {\em factorially closed} in $B$ if, for any two nonzero elements $b_1, b_2 \in B$, $b_1 b_2 \in A$ implies that $b_1, b_2 \in A$. In some papers, a factorially closed subring is called an {\em inert} subring. It is well known that $B^{\sigma}$ is factorially closed in $B$ for any $\sigma \in \EXP (B)$  (see e.g., \cite[Lemma 2.1 (1)]{K16}).

\section{Retracts and rings of invariants of exponential maps}

In  this section, we give some sufficient conditions for a retract $A$ of $k^{[3]}$ to be a polynomial ring. Throughout this subsection, $R$ denotes a domain. 

\subsection{Retracts and rings of constants of exponential maps}

In this subsection, we consider retracts and rings of invariants of exponential maps on $R$-domains. 
The following result is a generalization of \cite[Proposition 4.2]{CDDG21} and \cite[Theorem 2.5]{LS21}.
In fact, the proof of Theorem 3.1 is almost the same as that of \cite[Theorem 2.5]{LS21}. 
In Theorem 3.1, the implication (1) $\Longrightarrow$ (2) and the equivalence of (1) and (3) when $B$ is a UFD follow from \cite[Proposition 4.2]{CDDG21}. Here, we use the technique given in the proof of \cite[Theorem 2.5]{LS21}.
\begin{thm}
Let $B$ be an $R$-domain and $A$ an $R$-subalgebra of $B$. Then the following conditions are equivalent to each other{\rm :}
\begin{itemize}
\item[(1)]
$B \cong A^{[1]}$ as an $A$-algebra.
\item[(2)]
$A$ is a retract of $B$ with $B = A \oplus Bh$ for some $h \in B \setminus \{ 0 \}$ and $A = B^{\sigma}$ for some non-trivial exponential map $\sigma$ on $B$.
\end{itemize}
Moreover, if $B$ is a UFD, then the condition {\rm (1)} is equivalent to the following condition {\rm (3)}.
\begin{itemize}
\item[(3)]
$A$ is a retract of $B$ and $A = B^{\sigma}$ for some non-trivial exponential map $\sigma$ on $B$. 
\end{itemize}
\end{thm}

\begin{proof}
The proof consists of three parts.

\noindent
\textbf{(1) $\Longrightarrow$ (2) and (1) $\Longrightarrow$ (3).} 
By (1), $B = A[h]$, where $h$ is an indeterminate. So $A$ is a retract of $B$ and $B = A \oplus Bh$. We have an $R$-algebra homomorphism $\sigma: B \to B^{[1]} = B[t]$ such that $\sigma(h) = h + t$ and  $\sigma(a) = a$ for any $a \in A$. Then $\sigma$ is an exponential map on $B$ and $B^{\sigma} = A$. This proves (2) and (3). 
\smallskip

\noindent
\textbf{(2) $\Longrightarrow$ (1).} 
Let $\pi: B \to A=B^{\sigma}$ be the projection form $B$ onto $A$ regarding to the decomposition $B = A \oplus Bh$. We know that $\pi$ is a retraction with $\Ker \pi = Bh$. Then $h$ is a local slice of $\sigma$. Indeed, for a local slice $s \in B$, $s':= s - \pi(s) \in Bh$, say $s' = bh$ for some $b \in B\setminus \{0 \}$. Since $h \not\in A = B^{\sigma}$, $\deg_{\sigma} (h) \geq \deg_{\sigma} (s)$. Moreover, 
$$
\deg_{\sigma} (s) = \deg_{\sigma} (s') = \deg_{\sigma} (b) + \deg_{\sigma} (h) \geq \deg_{\sigma} (h),
$$
which implies $\deg_{\sigma} (h) = \deg_{\sigma} (s)$. Hence $h$ is a local slice of $\sigma$. 

By Lemma 2.3, $B_a = A_a [h]$, where $a = \lc_{\sigma}(h) \in A$. Since $A$ is algebraically closed in $B$ and $h \not\in A$, $A [h] \cong A^{[1]}$ as $A$-algebras. For any $f \in B$, there exists a positive integer $m$ such that $a^m f \in A[h]$. We may assume that $a^m f \in M:= A \oplus Ah \oplus \cdots \oplus Ah^r$ for some $r \geq 0$. Since 
$$
B = A \oplus Bh = A \oplus (A \oplus Bh) h = A \oplus Ah \oplus Bh^2,
$$ 
we have $B = M \oplus Bh^{r+1}$. If $f = g + bh^{r+1}$ for some $g \in M$ and $b \in B$, then $a^m b h^{r+1} = a^m f - a^m g \in M \cap Bh^{r+1} = \{ 0 \}$, hence $b = 0$. Therefore, $f = g  \in M \subset A[h]$. This proves (1). 
\smallskip

\noindent
\textbf{(3) $\Longrightarrow$ (2) in the case where $B$ is a UFD.} Let $\pi: B \to A$ be a retraction and set $I = \Ker \pi$. Then $B = A \oplus I$. Take a local slice $s$ of $\sigma$, which exists since $\sigma$ is non-trivial. Since $s - \pi(s) \in I$ and $s - \pi(s)$ is a local slice of $\sigma$ because $\pi(s) \in A = B^{\sigma}$, we may assume that $s \in I$. Let $s = p_1 \cdots p_{\ell}$ be the prime decomposition of $s$.
Since
$
\deg_{\sigma} (s) = \sum_{i=1}^{\ell} \deg_{\sigma}(p_i)
$
and $\deg_{\sigma} (p_i) = 0$ or $\geq \deg_{\sigma} (s)$ for $i = 1, \ldots, \ell$, there exists unique $i$ such that $\deg_{\sigma} (p_i ) = \deg_{\sigma} (s)$, say $\deg_{\sigma} (p_1) = \deg_{\sigma} (s)$, and $\deg_{\sigma} (p_2) = \cdots = \deg_{\sigma} (p_{\ell}) = 0$. 
Namely, $p_1$ is a local slice of $\sigma$ and $p_2, \ldots, p_{\ell} \in A$. Set $a = \lc_{\sigma} (p_1)$. By Lemma 2.3, $B_{a} = A_{a}[p_1]$ and $p_1$ is indeterminate over $A$. 
Since $s = p_1 p_2 \cdots p_s \in I$ and $p_2, \ldots, p_s \in A$, $p_1 \in I$. 

For any $u \in I$, there exist a positive integer $n(u)$ and some $a_0, a_1, \ldots, a_n \in A$ such that
$$
a^{n(u)} u = a_0 + a_1 p_1 + \cdots + a_n p_1^n. 
$$
Since $u, p_1 \in I$, $a_0 \in I \cap A$. Then $a_0 = 0$ and so $a^{n(u)} u \in p_1 B$. Since $p_1$ is a prime element of $B$, $u \in p_1 B$ or $a^{n(u)} \in p_1 B$. If $a^{n(u)} \in p_1 B$, then $v p_1 = a^{n(u)}$ for some $v \in B$. Since $a \in A = B^{\sigma}$ and $A$ is factorially closed in $B$, $p_1 \in A$. This is a contradiction. Therefore, $I = p_1 B$ is a principal ideal. This proves (2). 
\end{proof}

By Theorem 3.1 and the cancellation theorem for the affine plane over a field, we have the following result.

\begin{cor}
Let $k$ be a field and $A$ a $k$-algebra. If $A$ is a retract of $B = k^{[3]}$ and there exists a non-trivial exponential map $\sigma$ on $B$ such that $A = B^{\sigma}$, then $A \cong k^{[2]}$.
\end{cor}

\begin{proof}
By Theorem 3.1, $B \cong A^{[1]}$ as $k$-algebras. By the cancellation property of $k^{[2]}$ (see \cite[Theorem 2.1]{BG15}, \cite[Theorem 1.2]{Kuroda17} or \cite[Theorem 3.4]{K16}), we have $A \cong k^{[2]}$. 
\end{proof}

\subsection{Retracts of $k^{[3]}$ with non-trivial exponential maps}

In this subsection, we prove the following result. 

\begin{thm}
Let $k$ be a field of $p:= \ch (k) \geq 0$ and $B = k^{[3]}$ the polynomial ring in three variables over $k$. Let $A$ be a retract of $B$ with $\trdeg_k A = 2$. If $p = 0$ or $\ML(A) \not= A$, then $A \cong k^{[2]}$ as $k$-algebras. 
\end{thm}

\begin{proof}
If $p = 0$, then $A \cong k^{[2]}$ by \cite[Theorem 5.8]{CDDG21} or \cite[Theorem 2.5]{N19}. 
From now on, assume that $p > 0$ and $\ML(A) \not= A$. Since $A$ is a retract of $B = k^{[3]}$ and $B$ is finitely generated as a $k$-algebra, $A$ is finitely generated as a $k$-algebra by Lemma 2.2 (2). Further, since $B$ is a UFD, so is $A$ by Lemma 2.2 (3). In particular, $A$ is normal. 

By $\ML(A) \not= A$, there exists a non-trivial exponential map $\sigma$ on $A$. Set $R = A^{\sigma}$. 

\begin{claim1}
$R \cong k^{[1]}$. 
\end{claim1}

\begin{proof}
Since $\sigma$ is a non-trivial exponential map on $A$, we have
$$
\trdeg_k R = \trdeg_k A - 1 = 2 -1 = 1.
$$
Since $A$ is a UFD, so is $R$ by \cite[Lemma 2.1 (1)]{K16}. Then we infer from \cite[Theorem III]{EZ70} (or \cite[Theorem (4.1)]{AEH72}) that $R \cong k^{[1]}$. 
\end{proof}

\begin{claim2}
{\rm 
$A$ is geometrically factorial over $k$, i.e., $A \otimes_k k'$ is a UFD for any algebraic extension field $k'$ of $k$.  
}
\end{claim2}

\begin{proof}
Let $k'$ be an algebraic extension field of $k$.
By \cite[Lemma 1.2]{N19}, $A \otimes_k k'$ is a retract of $B \otimes_k k'$ over $k'$. 
Since $B \otimes_k k' \cong k'^{[3]}$ is a UFD, so is $A \otimes_k k'$. 
This proves the assertion. 
\end{proof}

\begin{claim3}
{\rm 
$A \cong k^{[2]}$. 
}
\end{claim3}

\begin{proof}
By Claim 1, $R = k[f]$ for some $f \in A \setminus k$. 
By \cite[Lemma 2.1 (3)]{K16}, $Q(R) \cap A = R$. Hence, $k(f) \cap A = k[f]$. 
Set $S= k[f] \setminus \{ 0 \}$. 
Then, by Lemma 2.3, we know that $S^{-1} A \cong k(f)^{[1]}$ because $R = k[f]$. 
By Claim 2, $A$ is geometrically factorial over $k$. Therefore, by virtue of Lemma 3.4 below, we know that $A \cong k[f]^{[1]} \cong k^{[2]}$.  
\end{proof}

Theorem 3.3 is thus verified. 

\end{proof}

\begin{lem} 
{\rm (Russell--Sathaye \cite[Theorem 2.4.2]{RS79})}
Let $k$ be a field, $A$ a finitely generated $k$-domain and $f \in A$. Set $S = k[f] \setminus \{ 0 \}$. Assume that the following conditions {\rm (1)}--{\rm (3)} are satisfied{\rm :}
\begin{enumerate}
\item[{\rm (1)}]
$S^{-1} A = k(f)^{[1]}$.
\item[{\rm (2)}]
$k(f) \cap A = k[f]$.
\item[{\rm (3)}]
$A$ is geometrically factorial over $k$. 
\end{enumerate}
Then $A \cong k[f]^{[1]}$. 
\end{lem}

\begin{proof}
See \cite[2.4]{RS79}, where we consider $L$ and $F$ as $k$ and $f$, respectively.
\end{proof}

Let $A$ be a $k$-subalgebra of $k^{[n]}$ and assume that $A$ is factorially closed in $k^{[n]}$. We easily see that $A$ is a polynomial ring if $k$ is algebraically closed and $\trdeg_k A\in \{0,1,n\}$. In particular, in \cite[Theorem 1 (2)]{CGM15}, Chakraborty, Gurjar and Miyanishi proved that if $k$ is algebraically closed of characteristic zero, then every factorially closed subring of $k^{[3]}$ is a polynomial ring. By the argument of the proof of Theorem 3.3, we obtain the following result.

\begin{prop}
Let $k$ be an algebraically closed field of $p:= \ch (k) \geq 0$ and $B = k^{[3]}$ the polynomial ring in three variables over $k$. Let $A$ be a factorially closed $k$-subalgebra of $B$ with $\trdeg_k A = 2$. If $p = 0$ or $\ML(A) \not= A$, then $A \cong k^{[2]}$ as $k$-algebras. 
\end{prop}

\begin{proof}
If $p = 0$, then $A \cong k^{[2]}$ by \cite[Theorem 1 (2)]{CGM15}. So we assume that $p > 0$ and $\ML(A) \not= A$. Since $A$ is algebraically closed in $B$, $A = Q(A) \cap B$. By  $\trdeg_k A = 2 (<3)$ and Zariski's theorem in \cite{Z54}, we know that $A$ is finitely generated over $k$. By $\ML(A) \not= A$, there exists a non-trivial exponential map $\sigma$ on $A$. Then the proof of Claim 1 in the proof of Theorem 3.3 works in our setting. So $R:= A^{\sigma} \cong k^{[1]}$. Since $A$ is factorially closed in $B$ and $B$ is a UFD, $A$ is a UFD. In particular, $A$ is geometrically factorial since $k$ is algebraically closed. Finally, we easily see that the proof of Claim 3 in the proof of Theorem 3.3 works in our setting. This proves $A \cong k^{[2]}$.
\end{proof}

\subsection{Retracts of $k^{[3]}$ with $\Z$-gradings}

In this subsection, we consider retracts having $\Z$-gradings. 
Let $A$ be a finitely generated $k$-domain and $X=\Spec A$. $\G_m$ denotes the $1$-dimensional algebraic torus over $k$. It is well known that a $\G_m$-action on $X$ is corresponding to a $\Z$-grading on $A$, that is, $A=\bigoplus_{d\in\Z}A_d$, where $A_d$ denotes the homogeneous elements of $A$ of degree $d\in\Z$. 
A $\G_m$-action on $X$ is said to be {\it effective} if $A\not=A_0$ with respect to the corresponding $\Z$-grading on $A$. $X$ is called an {\it affine $\G_m$-variety} if $X$ has an effective $\G_m$-action. In \cite{K23}, the first author proved that the following. 
\begin{thm} {\rm (\cite[Theorem 1.2]{K23})}\label{G_m-surf}
Let $k$ be an algebraically closed field of any characteristic and $S=\Spec A$ be an affine $\G_m$-surface over $k$. If $S$ is smooth and $A$ is a UFD with $A^*=k^*$, then $S\cong\A^2$. That is, $A\cong k^{[2]}$. 
\end{thm}

In \cite[Corollary 5.15]{CDDG21}, Chakraborty, Dasgupta, Dutta and Gupta proved that if $A$ is a graded subalgebra of $B\cong k^{[n]}$ with the standard grading and there exists a retraction $\pi:B\to A$ such that $\pi(B_+)\subseteq B_+$, then $A$ is a polynomial ring over $k$. The following theorem suggests that these assumptions are unnecessary when $n=3$. 

\begin{thm} \label{graded}
Let $k$ be an algebraically closed field of any characteristic and $A$ be a retract of $B\cong k^{[n]}$ for some $n\geq2$ with $\trdeg_k\:A=2$. If $A$ has a $\Z$-grading with $A_0\not=A$, then $A\cong k^{[2]}$. Therefore, every $\Z$-graded ring which is a retract of $k^{[3]}$ is a polynomial ring over $k$. 
\end{thm}

\begin{proof}
Let $S:=\Spec A$. By the assumption $S$ is an affine $\G_m$-surface. Since $A$ is a retract of $B\cong k^{[n]}$, Lemma \ref{properties of retracts} implies that A is a regular UFD with $A^*=k^*$. In particular, $S$ is smooth. Therefore, it follows from Theorem \ref{G_m-surf} that $A\cong k^{[2]}$. 
\end{proof}

\section{Examples of retracts of low-dimensional polynomial rings}

In this section, we give some examples of retracts of polynomial rings in two or three variables and we determine their generators. Additionally, we check whether each retract is a polynomial ring or not. 

Let $k$ be a field and $A$ a retract of $B=k^{[n]}=k[x_1, \ldots, x_n]$ for a positive integer $n$. 
It is clear that, if $n = 1$, then either $A=B$ or $A = k$.
If either $n=2$ or $n=3$ and $\ch(k) = 0$, then we know that $A$ is a polynomial ring over $k$. 
Furthermore, $\trdeg_k A =1$ (resp.\ $n=3$, $\ch(k) = 0$ and $\trdeg_k A = 2$), then $A=k[f]$ (resp.\ $A=k[f_1,f_2]$) for some $f \in B \setminus k$ (resp.\ $f_1,f_2 \in B \setminus k$). 
When $n = 2$ and $\ch(k) = 0$,  Spilrain--Yu \cite{SY00} gave a characterization of all the retracts of $B$ up to an automorphism, and gave several applications of this characterization to the $2$-dimensional Jacobian conjecture. However, it is difficult to determine the retracts of $B$ when $n \ge 2$. 

Let $\varphi$ be a retraction from $B$ to $A$.
Say, $A = \operatorname{Im} \varphi$, where $\varphi : B \to B$ is a $k$-algebra endomorphism of $B$ and $\varphi ^{2} = \varphi$.
We set $f_i=\varphi(x_i)$ for $i=1, \dots ,n$.
Since $\varphi$ is a $k$-algebra homomorphism, we have
$$
\varphi(g) = g(\varphi(x_1), \dots ,\varphi(x_n)) = g(f_1, \dots ,f_n)
$$
for any polynomial $g \in B$. 
In particular,
\begin{equation}
f_i=\varphi(x_i)=\varphi (\varphi(x_i)) = \varphi(f_i) =f_i(f_1, \dots ,f_n) \tag{4.1}
\end{equation}
for each $i=1, \dots ,n$. Furthermore, $A=k[f_1, \dots ,f_n]$.

Conversely, suppose $A=k[f_1, \dots ,f_n]$ is a $k$-algebra of $B$ generated by $n$ polynomials $f_1, \dots ,f_n$. 
We set the $k$-algebra homomorphism $\varphi : B \to B$ given by $\varphi(x_i)=f_i$ for each $i=1, \dots ,n$.
If the $f_i$'s satisfy (4.1), then $\varphi$ is a retraction.
Thus we have the following lemma.

\begin{lem}
With the same notations as above, $\varphi$ is a retraction if and only if the $f_i$'s satisfy {\rm (4.1)}.
\end{lem}

Now, we give some examples of retracts of polynomial rings by using Lemma 4.1. In this section, we consider the cases $n=2, 3$ only. 
We note that some of the results of this section can be generalized. However, we do not give such a generalization since even in the low dimensional case has not been fully understood.
We sometimes set $k^{[2]}=k[x,y]$ and $k^{[3]}=k[x,y,z]$.

In the following example, we determine the retracts of $B=k^{[n]} \ (n=2,3)$ generated by monomials. Here, we assume that a monomial is monic. 

\begin{example}
{\rm 
Let $f_1, \dots ,f_n$ be $n$ polynomials of $B=k^{[n]}$ that satisfy {\rm (4.1)} and let $A=k[f_1, \dots ,f_n]$.
Assume that $A \ne k$ and every $f_i \ (i=1, \dots ,n)$ is either $0$ or a monomial.
Then there exists a positive integer $r$ such that $1\leq r\leq n$  and $f_{m_1},\ldots,f_{m_r}$ are non-zero monomials and $f_{j}=0$ for $j\not\in\{m_1,\ldots,m_r\}$. For $1\leq i\leq r$, let $f_{m_i}=x_1^{a_{m_i1}}\cdots x_n^{a_{m_in}}$ be a monomial, where $a_{m_ij}\geq 0$. Then the $r\times r$ matrix $(a_{m_im_j})_{1\leq i,j,\leq r}$ is idempotent. Indeed, by (4.1), $a_{m_ij}=0$ when $f_j=0$, thus $f_{m_i}\in k[x_{m_1},\ldots,x_{m_r}]$. Therefore, (4.1) implies that $(a_{m_im_j})_{1\leq i,j,\leq r}$ is an idempotent matrix. By simple calculation, the following assertions hold.
\begin{enumerate}
\item[(1)]
If $n=2$, then $f_1$ and $f_2$ are as in the following table, where $m$ is a non-negative integer.

\begin{longtable}{cc|c}
$f_1$         &               $f_2$         &          $A=k[f_1,f_2]$ \\
\hline
$x$           &               $0$          &          $k[x]$ \\
$0$           &               $y$          &          $k[y]$ \\
$x$           &               $x^{m}$     &         $k[x]$ \\
$x$           &               $y$          &         $k[x,y]$ \\
$xy^{m}$    &               $1$          &         $k[xy^{m}]$ \\
$y^{m}$     &                $y$          &         $k[y]$ \\
$1$          &               $x^{m}y$    &        $k[x^{m}y]$ \\ 
\end{longtable}

\item[(2)]
If $n=3$, then $f_1$, $f_2$ and $f_3$ are as in the following table, where $m$ and $l$ are non-negative integers.
\vspace{1cm}

\begin{longtable}{ccc|c}
$f_1$          &          $f_2$             &  $f_3$          & $A=k[f_1,f_2,f_3]$ \\
\hline
$x$            &           $0$              &  $0$            &  $k[x]$ \\ 
$0$            &           $y$              &  $0$            &  $k[y]$ \\
$0$            &           $0$              &  $z$            &  $k[z]$ \\ 

$x$            &           $x^{m}$         &  $0$        &  $k[x]$ \\
$x$            &           $y$              &  $0$        &  $k[x,y]$ \\
$xy^{m}$      &          $1$              &  $0$        &  $k[xy^{m}]$ \\
$y^{m}$       &           $y$              &  $0$        &  $k[y]$ \\
$1$            &           $x^{m}y$       &  $0$        &  $k[x^{m}y]$ \\

$x$            &           $0$              &  $x^{m}$        &  $k[x]$ \\
$x$            &           $0$              &  $z$             &  $k[x,z]$ \\
$xz^{m}$     &           $0$              &  $1$            &  $k[xz^{m}]$ \\
$z^{m}$       &           $0$              &  $z$            &  $k[z]$ \\
$1$            &           $0$              &  $x^{m}z$      &  $k[x^{m}z]$ \\

$0$            &           $y$              &  $y^{m}$    &  $k[y]$ \\
$0$            &           $y$              &  $z$         &  $k[y,z]$ \\
$0$            &           $yz^{m}$       &  $1$         &  $k[yz^{m}]$ \\
$0$            &           $z^{m}$         &  $z$         &  $k[z]$ \\
$0$            &           $1$              &  $y^{m}z$   &  $k[y^{m}z]$ \\

$x$                &           $x^{l}$                &  $x^{m}$         &  $k[x]$ \\
$y^{l}$              &           $y$                 &  $y^{m}$          &  $k[y]$ \\
$z^{l}$            &           $z^{m}$              &  $z$                 &  $k[z]$ \\

$x$                &           $y$                   &  $x^{l}y^{m}$     &  $k[x,y]$ \\
$y^{l}z^{m}$       &           $y$                 &  $z$                &  $k[y,z]$ \\
$x$                 &           $x^{l}y^{m}$       &  $z$                &  $k[x,z]$ \\

$x$                 &           $y$                  &  $z$               &  $k[x,y,z]$ \\

$xy^{l}$           &           $1$                   &  $x^{m}y^{lm}$    &  $k[xy^{l}]$ \\
$xz^{l}$            &           $x^{m}z^{lm}$     &  $1$               &  $k[xz^{l}]$ \\
$y^{l}z^{lm}$      &           $yz^{m}$           &  $1$                &  $k[yz^{m}]$ \\
$y^{lm}z^{l}$      &           $1$                  &  $y^{m}z$         &  $k[y^{m}z]$ \\
$1$               &           $x^{l}y$               &  $x^{lm}y^{m}$    &  $k[x^{l}y]$ \\
$1$               &           $x^{lm}z^{l}$       &  $x^{m}z$          &  $k[x^{m}z]$ \\

$xy^{l}$           &           $1$                   &  $y^{m}z$        &  $k[xy^{l},y^{m}z]$ \\
$xz^{l}$            &           $yz^{m}$           &  $1$               &  $k[xz^{l},yz^{m}]$ \\
$1$               &           $x^{l}y$               &  $x^{m}z$          &  $k[x^{l}y, x^{m}z]$ \\

$xy^{l}z^{m}$     &           $1$                  &  $1$               &  $k[xy^{l}z^{m}]$ \\
$1$               &           $x^{l}yz^{m}$        &  $1$                 &  $k[x^{l}yz^{m}]$ \\
$1$               &           $1$                    &  $x^{l}y^{m}z$     &  $k[x^{l}y^{m}z]$ \\
\end{longtable}
\end{enumerate}
In these cases, $A$ is a polynomial ring.
}
\end{example} 

We give elementary results on sufficient conditions for a retract of $k^{[3]}$ to be a polynomial ring in Propositions 4.4 and 4.6. In their proofs, we use the following.

\begin{lem} 
Let $f_1, f_2 ,f_3$ be three polynomials of $B=k^{[3]} = k[x_1, x_2, x_3]$ and let $A=k[f_1, f_2 ,f_3]$. Assume that $f_1, f_2, f_3$ satisfy {\rm (4.1)}  and one of the following conditions. 
\begin{enumerate}
\item[{\rm (1)}]
$f_1,f_2,f_3$ are algebraically independent over $k$. 
\item[{\rm (2)}]
There exists $i\in\{1,2,3\}$ such that $f_i\in k$. 
\item[{\rm (3)}]
There exists $i\in\{1,2,3\}$ such that $f_i=x_i$. 
\end{enumerate}
Then $A$ is a polynomial ring. 
\end{lem}

\begin{proof}
Suppose that the condition (1) holds. Then $\trdeg_kA=3$, hence $A=B=k^{[3]}$.

Suppose that the condition (2) holds. We may assume that $f_3\in k$. Then $A=k[f_1,f_2]$. If $f_1$ and $f_2$ are algebraically independent over $k$, then $\trdeg_k A = 2$ and hence $A\cong k^{[2]}$. Otherwise, since $\trdeg_kA \leq 1$, the assertion follows from \cite[Theorem 3.5]{C77}. 

Suppose that the condition (3) holds. We may assume that $f_3=x_3$. Let $R=k[x_3]$. Then $B=R[x_1,x_2]\cong R^{[2]}$, $A=R[f_1,f_2]$ and $\varphi$ is a homomorphism of $R$-algebras. Thus, $A$ is an $R$-algebra retract of $B\cong R^{[2]}$. Since $R$ is a UFD, it follows from \cite[Theorem 3.5]{C77} that $A\cong R^{[e]}$ for some $0\leq e\leq 2$. Therefore, $A\cong k^{[e+1]}$. 
\end{proof}

Here we prove the following.

\begin{prop}
Let $f_1, f_2 ,f_3$ be three polynomials of $B=k^{[3]}$ that satisfy {\rm (4.1)} and let $A=k[f_1, f_2 ,f_3]$.
Assume that $A \ne k$ and at least one of $f_1, f_2, f_3$ is a monomial. 
Then $A$ is a polynomial ring. 
\end{prop}

\begin{proof}
We may assume that $f_1$ is a monomial. 
If at least one of $f_1$, $f_2$, $f_3$ is an element of $k$, then the assertion follows from Lemma 4.3. So we may assume that $f_1$, $f_2$, $f_3$ are non-constant.
By the assumption, we set $f_1 = x^{p}y^{q}z^{r}$, where $p$,$q$,$r$ are non-negative integers and $(p,q,r) \ne (0,0,0)$.

Assume that neither $f_2$ nor $f_3$ is a monomial. 
If $q \geq 1$, $f_1(f_1,f_2,f_3)$ is not a monomial because $f_2$ is not a monomial.
So we have $f_1(f_1,f_2, f_3) \ne f_1$, which is a contradiction.
Similarly, we derive a contradiction if $r \geq 1$. 
So $q=0$ and $r=0$.
Then $f_1=x^{p}$ and $f_1(f_1,f_2,f_3)=x^{p^{2}}$.
Since $f_1$ satisfies the condition (4.1), we have $p=1$ and $f_1=x$.
By Lemma 4.3, $A$ is a polynomial ring. 

Assume that $f_2$ is a monomial and $f_3$ is not a monomial.
We set $f_2 = x^{s}y^{t}z^{u}$, where $s$,$t$,$u$ are non-negative integers and $(s,t,u) \ne (0,0,0)$.
By the same argument as above, we have $r=u=0$.
Then we have $(f_1,f_2) =(x,x^{m}), (x,y),(y^{m},y)$ by Example 4.2 (1).
We know that $A$ is a polynomial ring by Lemma 4.3.
\end{proof}

In Proposition 4.4, we can determine the $f_1, f_2, f_3$ under some assumptions. See the list of Example 4.2 (2), where every $f_i $ ($i = 1,2,3$) is either $0$ or a monomial. In fact, we have the following.

\begin{example}
{\rm 
Let $f_1, \dots ,f_n \in B \setminus k$ be $n$ polynomials of $B=k^{[n]}$ that satisfy {\rm (4.1)} and let $A=k[f_1, \dots ,f_n]$.
Assume further that the following conditions hold:
\begin{enumerate}
\item[(i)]
The constant term of $f_i$ is equal to $0$ for each $i=1, \dots ,n$.
\item[(ii)]
If $f_i \neq 0$, then its leading coefficient with respect to the lexicographical monomial order with $x > y > z$ is equal to $1$.
\end{enumerate}
In fact, for every retract of $B$, we can take its generator satisfying the conditions (i) and (ii) as above. 
By simple computation,  we have:

\begin{enumerate}
\item[(1)]
In case $n=2$, we assume that $f_1$ is a monomial and $f_2$ is not a monomial.
Then $f_1$ and $f_2$ are as in the following table.

\begin{longtable}{cc|c}
$f_1$         &               $f_2$         &          $A=k[f_1,f_2]$ \\
\hline
$x$           &               $f_2 \in k[x]$          &      $k[x]$ \\
\end{longtable}

\item[ (2)]
In case $n=3$, we assume both $f_1$ and $f_2$ are monomials and $f_3$ is not a monomial.
Then $f_1$, $f_2$ and $f_3$ are as in the following table, where $m$ is a non-negative integer and $g$ is an element of $B$.

\begin{longtable}{
ccc|c}
$f_1$                       &          $f_2$               &  $f_3$          & $A=k[f_1,f_2,f_3]$ \\
\hline
$x$            &           $x^{m}$             &  $f_3 \in k[x]$   &  $k[x]$ \\ 
$x$            &           $x^{m}$          &     $(x^{m}-y)g+z$   &  $k[x,f_3]$ \\
$x$            &           $y$              &     $f_3 \in k[x,y]$   &  $k[x,y]$ \\
$y^{m}$            &     $y$              &     $f_3 \in k[y]$     &  $k[y]$ \\
$y^{m}$            &      $y$     &            $(x-y^{m})g+z$            &  $k[y,f_3]$ \\ 
\end{longtable}
\end{enumerate}
}
\end{example}

In Proposition 4.4, we have not determined the $f_1, f_2, f_3$ when only one of them is a monomial. Of course, there are many such examples. 

Finally, we give the following result.

\begin{prop}
Let $f_1, f_2 ,f_3$ be three polynomials of $B=k^{[3]} = k[x_1, x_2, x_3]$ that satisfy {\rm (4.1)} and let $A=k[f_1, f_2 ,f_3]$.
Assume that $A \ne k$ and  $f_1, f_2 ,f_3$ satisfy the conditions {\rm (i)} and {\rm (ii)} in Example {\rm 4.5}.
If there exists $i \in \{1, 2 ,3 \}$ such that $f_i$ is a binomial and $x_i$ is a term of $f_i$, then $A$ is a polynomial ring.
\end{prop}

\begin{proof}
We may assume that $f_1$ is a binomial and $x_1$ is a term of $f_1$.
We may set $f_1=x_1 + \alpha x_1^{p}x_2^{q}x_3^{r}$, where $\alpha \in k \setminus \{0\}$, $p,q,r \in \Z_{\ge 0}$ and $(p,q,r) \ne (0,0,0)$.

If $q=0$ and $r=0$, then $f_1=x_1+\alpha x_1^{p}$.
Since $f_1$ is a binomial and the constant term of $f_1$ is equal to $0$ by the condition (i) in Example 4.5, we have $p \ge 2$.
So we know that $f_1(f_1,f_2,f_3)=f_1+\alpha f_1^{p}$.
By (4.1), we have $\alpha f_1^{p} =0$, which is a contradiction.
Therefore, we have $q \ne 0$ or $r \ne 0$.

Now we know that $f_1(f_1,f_2,f_3)=f_1+\alpha f_1^{p}f_2^{q}f_3^{r}$.
By the same argument as above, we have $\alpha f_1^{p}f_2^{q}f_3^{r} =0$.
Since $\alpha$ and $f_1$ are not equal to $0$ and $q \ne 0$ or $r \ne 0$, we have $f_2=0$ or $f_3=0$.
Thus we know that $A$ is a polynomial ring by Lemma 4.3.
\end{proof}



\begin{thebibliography}{30}
\bibitem{AEH72}
S. Abhyankar, P. Eakin and W. Heinzer, On the uniqueness of the coefficient ring in a polynomial ring, J. Algebra, \textbf{23} (1972), 310--342.

\bibitem{BG15}
S. M. Bhatwadekar and N. Gupta, A note on the cancellation property of $k[X,Y]$, J. Algebra Appl., \textbf{14} (2005), 1540007 (5 pages). 

\bibitem{CDDG21}
S. Chakraborty, N. Dasgupta, A. K. Dutta and N. Gupta, Some results on retracts of polynomial rings, J. Algebra, \textbf{567} (2021), 243--268. 

\bibitem{CGM15}
S. Chakraborty, R. V. Gurjar and M. Miyanishi, Factorially closed subrings of commutative rings, Algebra Number Theory, \textbf{9} (2015), 1137--1158. 

\bibitem{C77}
D. L. Costa, Retracts of polynomial rings, J. Algebra, \textbf{44} (1977), 492--502.

\bibitem{EZ70}
A. Evyatar and A. Zaks, Rings of polynomials, Proc.\ Amer.\ Math.\ Soc., \textbf{25} (1970), 559--562.

\bibitem{EMS136}
G. Freudenburg, Algebraic Theory of Locally Nilpotent Derivations (second edition), Encyclopaedia of Mathematical Sciences vol.\ 136, Invariant Theory and Algebraic Transformation Groups VII, Springer-Verlag, 2017. 

\bibitem{G14-1}
N. Gupta, On the cancellation problem for the affine space $\A^3$ in characteristic $p$, Invent.\ Math., \textbf{195} (2014), 279--288.

\bibitem{G14-2}
N. Gupta, On Zariski's cancellation problem in characteristic $p$, Adv.\ Math., \textbf{264} (2014), 296--307.

\bibitem{GN23}
N. Gupta and T. Nagamine, Retracts of Laurent polynomial rings, preprint (arXiv:2301.12681v3).

\bibitem{K75}
T. Kambayashi, On the absence of non-trivial separable forms of the affine plane, J. Algebra, \textbf{35} (1975), 449--456. 

\bibitem{K80}
T. Kambayashi, On Fujita's cancellation theorem for the affine plane, J. Fac.\ Sci.\ Univ.\ Tokyo, Sect IA, Math., \textbf{27} (1980), 535--548. 

\bibitem{K16}
H. Kojima, Notes on the kernels of locally finite iterative higher derivations in polynomial rings, Comm. Algebra, \textbf{44} (2016), 1924--1930.

\bibitem{K23}
H. Kojima, Smooth affine $\mathbb{G}_m$-surfaces with finite Picard groups and trivial units, Tokyo J. Math., {\bf 46} (2023), 93--109.

\bibitem{Kuroda17}
S. Kuroda, A generalization of Nakai's theorem on locally finite iterative higher derivations, Osaka J. Math., \textbf{54} (2017), 335--341.

\bibitem{LS18}
D. Liu and X. Sun, A class of retracts of polynomial rings, J. Pure Apple.\ Algebra, \textbf{222} (2018), 382--386. 

\bibitem{LS21}
D. Liu and X. Sun, Retracts that are kernels of locally nilpotent derivations, Czechoslovak Math.\ J., \textbf{72 (147)} (2021), 191--199. 

\bibitem{N19}
T. Nagamine, A note on retracts of polynomial rings in three variables, J. Algebra, \textbf{534} (2019),339--343. 

\bibitem{RS79}
P. Russell and A. Sathaye, On finding and cancelling variables in $k[X,Y,Z]$, J. Algebra \textbf{57} (1979), 151--166.

\bibitem{SY00}
V. Spilrain and J.-T. Yu, Polynomial retracts and the Jacobian conjecture, Trans.\ Amer.\ Math.\ Soc., \textbf{352} (2000), 477--484. 

\bibitem{Z54}
O. Zariski, Interpr\'{e}tations alg\'{e}brico-g\'{e}om\'{e}triques du quatorzi\`{e}me probl\`{e}me de Hilbert, Bull. Sci. Math., {\bf 78} (1954), 155--168.
\end{thebibliography}
\end{document}